\documentclass[11pt]
{article}

\newsavebox{\savepar}

\usepackage{mathrsfs}
\usepackage{amsmath}
\usepackage{amsfonts}
\usepackage{euscript}
\usepackage{oldgerm}
\usepackage{eufrak}
\usepackage{amsthm}
\usepackage{rotating}
\usepackage{colortbl}

 \usepackage[top=1.2in,bottom=1.2in,left=1in,right=1in]{geometry}  
\makeindex

  \newtheorem{theorem}{Theorem}[section]

  \theoremstyle{definition}

  \theoremstyle{remark}
  \newtheorem{remark}[theorem]{Remark}

\newtheorem{proposition}[theorem]{Proposition}

\theoremstyle{definition}

\theoremstyle{remark}

\date {today}
\begin{document}

\newcommand{\namelistlabel}[1]{\mbox{#1}\hfil}
\newenvironment{namelist}[1]{%
\begin{list}{}
{
\let\makelabel\namelistlabel
\settowidth{\labelwidth}{#1}
\setlength{\leftmargin}{1.1\labelwidth}
}
}{%
\end{list}}

\newcommand{\inp}[2]{\langle {#1} ,\,{#2} \rangle}
\newcommand{\vspan}[1]{{{\rm\,span}\{ #1 \}}}

\newcommand{\R} {{\mathbb{R}}}

\newcommand{\B} {{\mathbb{B}}}
\newcommand{\C} {{\mathbb{C}}}
\newcommand{\N} {{\mathbb{N}}}
\newcommand{\Q} {{\mathbb{Q}}}
\newcommand{\LL} {{\mathbb{L}}}
\newcommand{\Z} {{\mathbb{Z}}}

\newcommand{\BB} {{\mathcal{B}}}

\title{ Discrete  Modified Projection Methods for   Urysohn Integral 
Equations with  Green's Function Type Kernels }
\author{  Rekha P.  KULKARNI \thanks{Department of Mathematics, I.I.T. Bombay, 
Powai, Mumbai 400076, India,  rpk@math.iitb.ac.in,   } and Gobinda RAKSHIT \thanks{
gobindarakshit@math.iitb.ac.in} 
}

\date {}
\maketitle

\begin{abstract}

In the present paper we consider   discrete versions of the  modified projection methods for solving a Urysohn integral equation 
with a kernel of the type of Green's function. For $r \geq 0,$ a space of piecewise polynomials of 
degree $\leq r $ with respect to an uniform partition is chosen to be the approximating space. We define a discrete orthogonal projection onto this space and 
 replace the Urysohn integral operator by a Nystr\"{o}m approximation. 
The order of convergence which we 
obtain for the discrete version indicates the choice of numerical quadrature which preserves 
the orders of convergence in the continuous modified projection methods. 
 Numerical results are given for a specific example.
\end{abstract}

\bigskip\noindent
Key Words : Urysohn integral operator,   Orthogonal projection,   Nystr\"{o}m Approximation, Green's kernel

\smallskip
\noindent
AMS  subject classification : 45G10, 65J15, 65R20

\date {}

\newpage
\setcounter{equation}{0}
\section {Introduction}

Let $\mathcal{X} = L^\infty [0, 1]$ and 
consider the following nonlinear  Urysohn integral equation 
\begin{equation}\label{eq:1.1}
 x (s) -   \int_0^1 \kappa (s, t, x (t))  d t = f (s), \;\;\; s \in [0, 1], \; x \in \mathcal{X},
\end{equation}
where $f \in \mathcal{X}$ and the kernel $\kappa (s, t, u)$ is a continuous Green's function type kernel.
We write the above equation as
\begin{equation}\label{eq:1.2}
x - \mathcal{K} (x) = f
\end{equation}and assume that it has a unique solution $\varphi.$ 
We are interested in approximate solutions of the above equation. 

For $ r \geq 0,$ let
 $ \mathcal{X}_n$   be a 
space of piecewise polynomials of degree $ \leq r $ with respect to a 
uniform partition of $ [0, 1]$ with $n$ subintervals each of length $ {h = \frac {1} {n} }.$ 
Let $\pi_n$ be the restriction to $L^\infty [0, 1]$ of the orthogonal projection from 
$L^2 [0, 1]$ to  $\mathcal{X}_n.$ Then in the classical Galerkin method,
(\ref{eq:1.2}) is approximated by
\begin{equation*}
\varphi_n^G - \pi_n \mathcal {K} (\varphi_n^G)=\pi_n f. 
\end{equation*}
The above projection method has been studied extensively in research literature. 
See Krasnoselsii \cite{Kra}, Krasnoselskii et al \cite{KraV}
and Krasnoselskii-Zabreiko \cite{KraZ}.

\noindent
The iterated Galerkin solution is defined by
$$\varphi_n^S = \mathcal{K} (\varphi_n^G) + f. $$
The following orders of convergence are proved in  Atkinson-Potra \cite{AtkP1}: 
 
 If $r = 0,$ then
 \begin{equation}\nonumber
\| \varphi_n^G - \varphi \|_\infty = O ( h ), \;\;\; 
\| \varphi_n^S  - \varphi \|_\infty = O ( h^{ 2  } ),
\end{equation} 
whereas if $r \geq 1,$ then
\begin{equation}\nonumber
\| \varphi_n^G - \varphi \|_\infty = O ( h^{  r+1} ), \;\;\; 
\| \varphi_n^S  - \varphi \|_\infty = O ( h^{ r+3  } ).
\end{equation} 
In Grammont-Kulkarni \cite{Gram2}, the following modified
projection method is proposed:
\begin{equation}\nonumber
\varphi_n^M - \mathcal {K}_n^M (\varphi_n^M)= f, 
\end{equation}
where
\begin{eqnarray}\nonumber
 \mathcal {K}_n^M (x) = \pi_n  \mathcal {K} (x) + \mathcal {K} (\pi_n x) - \pi_n \mathcal {K} (\pi_n x).
\end{eqnarray}
The iterated modified projection solution is defined as 
\begin{equation}\nonumber
\tilde{\varphi}_n^M = \mathcal {K} (\varphi_n^M) + f.
\end{equation}
The following orders of convergence are proved in  Grammont et al \cite{Gram3}: 
 
 If $r = 0,$ then
 \begin{equation}\label{eq:1.3}
\| \varphi_n^M - \varphi \|_\infty = O ( h^{  3} ), \;\;\; 
\| \tilde{\varphi}_n^M  - \varphi \|_\infty = O ( h^{ 4  } ),
\end{equation} 
whereas if $r \geq 1,$ then
\begin{equation}\label{eq:1.4}
\| \varphi_n^M - \varphi \|_\infty = O ( h^{  r+3} ), \;\;\; 
\| \tilde{\varphi}_n^M  - \varphi \|_\infty = O ( h^{ r+5  } ).
\end{equation} 
In practice, it is necessary to replace the integral in the definition of $ \mathcal{K}$ by a numerical quadrature formula. Also, the orthogonal projection $\pi_n$ needs to be replaced by a discrete orthogonal projection $Q_n.$ This gives rise to the
 discrete versions of the above methods. 
 It is of interest to choose the  quadrature formula appropriately so as to preserve the above orders of convergence. The discrete versions of the Galerkin and the iterated Galerkin methods are considered in Atkinson-Potra \cite {AtkP2}.  Our aim is to  investigate  the discrete versions of the modified projection and of the iterated modified projection methods. 
 
 The discrete versions of the Galerkin and the iterated Galerkin methods are considered in Atkinson-Potra \cite {AtkP2}. 
  They propose a  numerical quadrature formula 
 which takes into consideration the fact that the kernel $\kappa (s, t, u)$ lacks smoothness when $ s = t$
 and obtain
 the order of convergence of the discrete iterated Galerkin solution.
 
 We follow a different approach. We choose a uniform partition
 with $m = n p, \; p \in \N,$ subintervals. A composite quadrature formula associated with  this fine partition is then used to  replace the integrals in the definition of  $ \mathcal{K}$ and in the definition of the inner product.   Let $ {\tilde {h} = \frac {1} {m}.}$ Let $z_n^M$ and
 $\tilde{z}_n^M$ denote respectively the discrete modified projection solution and the discrete iterated modified projection solution.  We prove the following orders of convergence:
 
 If $r=0,$ then
 \begin{equation}\label{eq:1.5}
  \| z_n^M - \varphi \|_\infty = O ( \max \{\tilde{h}^2, h^3 \} ), \;\;\; 
\| \tilde{z}_n^M  - \varphi \|_\infty = O ( \max \{\tilde{h}^2, h^4 \} ),
\end{equation} 
whereas if $r \geq 1,$ then
\begin{equation}\label{eq:1.6}
\|{z}_n^M - \varphi \|_\infty = O \left (  \max \left \{ \tilde{h}^2, h^{ r+3} \right \}\right), \;\;\;  \|\tilde{z}_n^M - \varphi \|_\infty = O \left ( \max \left \{\tilde{h}^2,   
h^{r + 5} \right \}\right).
\end{equation} 
Note that if $r = 0$ and $\tilde{h}^2 \leq h^{4},$ that is, $m \geq n^2,$ then the orders of convergence in (\ref{eq:1.3}) are preserved.
If $r \geq 1$ and $\tilde{h}^2 \leq h^{r+5},$ then the orders of convergence in (\ref{eq:1.4}) are preserved.

Note that the term $\tilde{h}^2$ in the above estimates appear because of the discretization. If the kernel is smooth, then it is possible to choose a composite quadrature formula associated with the coarse partition with $n$ subintervals and with a precision $d.$ Then the term $\tilde{h}^2$ is replaced by $h^d$ and an appropriate 
choice of $d$ will preserve the orders of convergence in  (\ref{eq:1.3}) and  (\ref{eq:1.4}). However, in the case of the kernel 
of the type of Green's function, the error in the higher order quadrature rules also is only of the order of $h^2.$ Hence we 
need to choose a different partition for the  quadrature rule  which makes the proofs more involved. 
It is to be noted that even if $m > n,$ the size of the system of equations that need to be solved in order to compute 
$z_n^M$ remains $n (r + 1).$

The paper has been arranged in the following way. In Section 2, we define a discrete orthogonal projection operator and discrete versions of modified projection methods.  In Section 3, we consider the case of piecewise polynomial space of degree $r \geq 1$ and prove (\ref{eq:1.6}).  Section 4 is devoted  to the proof of (\ref{eq:1.5}) in the case of piecewise constant functions. 
Numerical results for illustrative purpose are given in Section 5. 

\setcounter{equation}{0}
\section {Discrete modified projection method }

In this section we describe the Nystr\"{o}m  approximation of $ \mathcal{K}$ and the discrete orthogonal projection. We then define discrete versions of the modified projection method and its iterated version.

\subsection{Kernel of the type of Green's function}\label{subsection:2.1}

Let $r \geq 0$   be an integer and
assume that the kernel $\kappa$ of the integral operator $\mathcal {K}$ defined in (\ref{eq:1.1}) has the following properties.

\begin{enumerate}
\item Let $ \Psi = [0, 1] \times [0, 1] \times \R.$ The partial derivative
$ \displaystyle {\ell (s, t, u) = \frac {\partial \kappa ( s, t, u) } { \partial u}}$
is continuous for all $ (s, t, u) \in \Psi. $

\item Let
$ \Psi_1 = \{ (s, t, u): 0 \leq t \leq s \leq 1, \; u \in \R \},\;\;\;
 \Psi_2 = \{ (s, t, u): 0 \leq s \leq t \leq 1, \; u \in \R \}.$
There are functions $\ell_i \in C^{r+1}  ( \Psi_i ), i = 1, 2, $ with
$$
\ell (s,t, u) = \left\{ {\begin{array}{ll}
\ell_1 (s, t, u), \;\;\; (s, t, u) \in \Psi_1,   \\
\ell_2 (s, t, u), \;\;\; (s, t, u) \in \Psi_2.
\end{array}}\right.
$$

\item
There are two functions $\kappa_i \in C^{r+1} ( \Psi_i ), i = 1, 2, $ such that
$$
\kappa (s,t, u) = \left\{ {\begin{array}{ll}
\kappa_1 (s, t, u), \;\;\; (s, t, u) \in \Psi_1,   \\
\kappa_2 (s, t, u), \;\;\; (s, t, u) \in \Psi_2.
\end{array}}\right.
$$

\item 
$ \displaystyle {\frac { \partial^2 \kappa} {\partial u^2} \in C ( \Psi).}$
\end {enumerate}
Following Atkinson-Potra \cite{AtkP1}, if the kernel  $ \kappa$ satisfies the above conditions,
then we say that $ \kappa$ is of class $\mathcal{G}_2 ({r+1}, 0).$

Let $ f \in C^{r+1} [0, 1]$, then by the Corollary 3.2 of Atkinson-Potra \cite{AtkP1}, it follows that $\varphi \in C^{r+1} [0, 1].$ If $ r = 0,$ then it is assumed that $ f \in C^{2} [0, 1]$ so that $\varphi \in C^{2} [0, 1].$ We assume that $\mathcal{K}$ is twice Fr\'echet  differentiable and that $1$ is not an eigenvalue of $\mathcal {K}' (\varphi).$

\subsection{Nystr\"{o}m  approximation}
Let $m \in \mathbb{N}$ and
 consider the following uniform partition of $[0, 1]:$
\begin{equation}\label{eq:2.1}
0  <  \frac{1} {m}  < \cdots <   \frac{m-1} {m}   <  1.
\end{equation}
Let 
$\displaystyle { \tilde{h} = \frac {1} {m} \;\; \mbox {and} \;\; s_i = \frac {i} {m},  \;\; i = 0, \ldots, m.}$
   Consider a basic quadrature rule of the form
  \begin{equation}\nonumber
  \int_0^1 f (t) d t \approx \sum_{q=1}^{\rho} w_q f (\mu_q),
  \end{equation}
  where the weights $w_q > 0$ and the nodes $\mu_q \in [0, 1].$ It is assumed that the 
  quadrature rule is exact 
   at least for polynomials of degree $\leq 2 r. $ 
   Then $\displaystyle{\sum_{q=1}^{\rho} w_q  =  1}.$ 
  
A composite integration rule with respect to the partition  (\ref{eq:2.1}) 
is then defined as
\begin{eqnarray}\label{eq:2.2}
 \int_0^1 f (t) d t &=& \sum_{i =1}^m \int_{s_{i -1}}^{s_i} f (t) d t 
  \approx  \tilde h \sum_{i=1}^m  \sum_{q =1}^{\rho} w_q \; f (\zeta_q^i ), \;\;\; \zeta_q^i = s_{i -1} + \mu_q \tilde{h}.
\end{eqnarray}
We replace the integral in (\ref{eq:1.1}) by the numerical quadrature formula (\ref{eq:2.2}) and define
the Nystr\"{o}m operator as
\begin{equation}\nonumber
\mathcal{K}_m (x) (s)  =  \tilde {h}  \sum_{i=1}^m  \sum_{q=1}^\rho w_q \;  \kappa \left (s,  \zeta_q^i, x \left (\zeta_q^i \right ) \right), \;\;\; s \in [0, 1].
\end{equation}
Note that $\mathcal{K}_m$ is twice Fr\'echet  differentiable and
\begin{equation}\nonumber
\mathcal {K}_m'  (x) v (s)  =  \tilde {h}  \sum_{i=1}^m  \sum_{q=1}^\rho  w_q \;  \frac {\partial \kappa } {\partial u} (s,  \zeta_q^i, x (\zeta_q^i)) v (\zeta_q^i), \;\;\; s \in [0, 1],
\end{equation}
and
\begin{equation}\nonumber
\mathcal {K}_m''  (x) v (s)  =  \tilde {h}  \sum_{i=1}^m  \sum_{q=1}^\rho  w_q \;  \frac {\partial^2 \kappa } {\partial u^2} (s,  \zeta_q^i, x (\zeta_q^i)) v (\zeta_q^i), \;\;\; s \in [0, 1].
\end{equation}
For $ \delta_0 > 0, $ let
$ \displaystyle {\mathcal{B}(\varphi ,\delta_0) = \{ \psi \in \mathcal{X}: \| \varphi - \psi \|_\infty < 
\delta_0 \}.}$
Define
$$  C_1 =   \max_{\stackrel {s, t \in [0, 1]}{|u| \leq \|\varphi \|_\infty + \delta_0 }}
\left | \frac {\partial \kappa } {\partial u} (s, t, u) \right | \; \mbox {and} \;\;\; C_2 =   \max_{\stackrel {s, t \in [0, 1]}{|u| \leq \|\varphi \|_\infty + \delta_0 }}
\left | \frac {\partial^2 \kappa } {\partial u^2} (s, t, u) \right |. $$
Then  for $x, y \in \mathcal{B}(\varphi ,\delta_0),$
\begin{eqnarray}\label{eq:2.3}
\|\mathcal {K}_m' (x) v \|_\infty \leq  C_1 \|v\|_\infty
\end{eqnarray}
and
\begin{eqnarray}\label{eq:2.4}
\|\mathcal {K}_m'  (x) - \mathcal {K}_m'  (y)\| \leq  C_2  \|x - y \|_\infty.
\end{eqnarray}
Even though the numerical quadrature is assumed to be exact for polynomials of degree $\leq 2 r,$ since the kernel
lacks smoothness along $s = t,$ we only have the following order of convergence from Atkinson-Potra \cite {AtkP2}:
If $x \in C^2 [0, 1],$ then
\begin{equation}\label{eq:2.5}
\|\mathcal{K} (x) - \mathcal{K}_m (x) \|_\infty = O \left (\tilde{h}^2 \right ).
\end{equation}
In the Nystr\"{o}m method,  (\ref{eq:1.2}) is approximated by
\begin{equation}\nonumber
x_m - \mathcal{K}_m (x_m) = f.
\end{equation}
For all $m$ big enough, the above equation   has a unique solution $\varphi_m$ in $\mathcal{B} (\varphi, \delta_0 )$
and
\begin{eqnarray}\label{eq:2.6}
\|\varphi - \varphi_m \|_\infty  \leq C_3  \|\mathcal{K} (\varphi) - \mathcal{K}_m (\varphi ) \|_\infty = O \left (\tilde{h}^{2}\right).
\end{eqnarray}
See Atkinson \cite{Atk1}.
We quote the following result from Krasnoselskii et al  \cite{KraV}  for future reference:

If $v_1, v_2 \in \mathcal{B} (\varphi, \delta_0),$ then by the generalized Taylor's theorem,
\begin{eqnarray}\label{eq:2.7}
\mathcal{K}_m (  v_2 ) (s) -  \mathcal{K}_m ( v_1 ) (s) - \mathcal{K}_m' ( v_1) (v_2 - v_1) (s) &= & 
 R  (v_2 - v_1)  (s), \; s \in [0, 1], 
 \end{eqnarray} 
 where
 \begin{eqnarray}\nonumber
&&R  (v_2 - v_1)  (s)
 = \int_0^1  {(1 - \theta) }   \mathcal {K}_m^{''} \left (v_1 + \theta  
 (v_2 - v_1) \right ) (v_2 - v_1)^2  (s)    d \theta.
\end{eqnarray}
It then follows that
\begin{equation}\label{eq:2.8}
\| R  (v_2 - v_1)  \|_\infty \leq  C_2  \|v_2 - v_1\|_\infty^2.
\end{equation}
\subsection {Discrete orthogonal projection}

Let $n \in \mathbb{N}$ and consider the following uniform partition of $[0, 1]:$
\begin{equation}\label{eq:2.9}
\Delta: 0  <  \frac{1} {n}  < \cdots <   \frac{n-1} {n}   <  1.
\end{equation}
Define
\begin{equation}\nonumber
t_j = \frac {j} {n}, \;\;\;\Delta_j = [t_{j-1}, t_j]   \;\;\; \mbox {and} \;\;\; h = t_{j} - t_{j-1} = \frac {1} {n}, \;\;\; j = 1, \ldots, n.
\end{equation}
For $r \geq 0,$ let $\mathcal{X}_n$ denote the space of piecewise polynomials of degree $\leq r $ with respect to 
the partition of (\ref{eq:2.9}) of $[0, 1].$ Assume that the values at 
$t_j-, \; j = 1, \ldots, n,$ are defined by continuity. Then the dimension of $\mathcal{X}_n$ is $ n (r+1).$

\noindent
For $\eta = 0, 1, \ldots, r,$ let $L_\eta $ denote the Legendre polynomial of degree $\eta$ on $[-1, 1].$ Define
\begin{eqnarray}\nonumber
\varphi_{1, \eta} (t) &=& \left\{ {\begin{array}{ll}
\sqrt {\frac {2} { h}} L_\eta \left ( \frac {2 t - t_1 - t_0} {h} \right ), \;\;\;  t \in [t_{0}, t_1],   \\
0, \;\;\;  \mbox{otherwise}. 
\end{array}}\right. 
\end{eqnarray}
For $j = 2, \ldots, n,$ and for  $\eta = 0, 1, \ldots, r,$ define
\begin{eqnarray}\nonumber
\varphi_{j, \eta} (t) &=& \left\{ {\begin{array}{ll}
\sqrt {\frac {2} { h}} L_\eta \left ( \frac {2 t - t_j - t_{j-1}} {h} \right ), \;\;\;  t \in (t_{j-1}, t_j],   \\
0, \;\;\;  \mbox{otherwise}. 
\end{array}}\right. 
\end{eqnarray}
Note that $\displaystyle { \|\varphi_{j, \eta} \|_\infty = \max_{t \in [t_{j-1}, t_j]} |\varphi_{j, \eta} (t)|  = \sqrt {\frac {2} { h}} \|L_\eta\|_\infty.}$

From now onwards we assume that $m = p n$ for some $p \in \N.$ Thus each interval $\Delta_j$ is
divided into $p$ equal parts and the integral over each interval $\Delta_j$ is approximated by using the quadrature formula (\ref {eq:2.2}) restricted to the interval $\Delta_j.$
For $f, g \in C (\Delta_j),$
define 
\begin{equation}\label{eq:2.10}
\inp {f} {g}_{\Delta_j} = \tilde {h} \sum_{\nu = 1}^p \sum_{q = 1}^{\rho}   w_q ~ f {\left (\zeta_q^{(j-1) p + \nu} \right )} ~ g{\left (\zeta_q^{(j-1) p + \nu} \right )},
\end{equation}
where $\zeta_q^{(j-1) p + \nu}$ are defined in (\ref{eq:2.2}).
Note that $\displaystyle {\inp {f} {g}_{\Delta_j} = \inp {g} {f}_{\Delta_j}.}$
Define
\begin{equation}\nonumber
\| f \|_{\Delta_j, \infty} = \max_{t \in [t_{j-1}, t_j]} |f (t)|.
\end{equation}
Then
\begin{equation}\label{eq:2.11}
\left |\inp {f} {g}_{\Delta_j} \right | \leq  \| f \|_{\Delta_j, \infty} \| g \|_{\Delta_j, \infty} h.
\end{equation}
Since the quadrature rule is exact for polynomials of degree $\leq 2 r ,$  it follows that
$$ \inp {\varphi_{j, \eta}}{\varphi_{j, \eta'}} = \int_{0}^{1}  \varphi_{j, \eta} (t) \varphi_{j, \eta'} (t) d t = 
  \inp {\varphi_{j, \eta}}{\varphi_{j, \eta'}}_{\Delta_j}.$$
Thus,
$\displaystyle {\{\varphi_{j, \eta},  j = 1, \ldots, n, \; \eta = 0,  \ldots, r \}}$
forms an orthonormal basis for  $\mathcal{X}_n \subset L^\infty [0, 1].$
Let $\mathcal{P}_{r, \Delta_j}$ denote the space of polynomials of degree $\leq r $  on $\Delta_j.$ 
Define the discrete orthogonal projection $Q_{n,j}: C (\Delta_j) \rightarrow \mathcal{P}_{r, \Delta_j}$ as follows:
\begin{equation}\nonumber
Q_{n, j} x = \sum_{\eta = 0}^r \inp {x} {\varphi_{j, \eta}}_{\Delta_j} \varphi_{j, \eta}. 
\end{equation} 
It follows that
\begin{eqnarray}\nonumber
\inp {Q_{n, j} x} {y}_{\Delta_j}  = \inp { x} {Q_{n, j} y}_{\Delta_j}, \;\;\; Q_{n, j}^2 = Q_{n, j} \;\;\; \mbox {and} \;\;\; Q_{n, j} Q_{n, i} = 0 \;\;\; \mbox {for} \;\;\; i \neq j.
\end{eqnarray}
Also,
\begin{eqnarray}\nonumber
\|Q_{n, j} x\|_{\Delta_j, \infty} &\leq& \sum_{\eta = 0}^r \left | \inp {x} {\varphi_{j, \eta}}_{\Delta_j} \right | 
\|\varphi_{j, \eta}\|_{\Delta_j, \infty}
 \leq   \left (2 \sum_{\eta = 0}^r  \|L_\eta\|_\infty^2 \right ) \|x\|_\infty.
\end{eqnarray}
A discrete orthogonal projection $Q_n:  C[0, 1] \rightarrow \mathcal{X}_n$ is defined as follows:
 \begin{eqnarray}\label{eq:2.12}
  Q_n x = \sum_{j=1}^n Q_{n, j} x.
  \end{eqnarray}
  Using the Hahn-Banach extension theorem, as in Atkinson et al \cite{AtkG}, $Q_n$ can be extended to $L^\infty [0, 1].$
 Then
\begin{eqnarray}\label{eq:2.13}
Q_n^2 = Q_n \;\;\;\mbox{and} \;\;\; \|Q_n\| \leq 2 \sum_{\eta = 0}^r  \|L_\eta\|_\infty^2 = C_4.
\end{eqnarray}
The following estimate is standard:
If $f \in C^{r+1} (\Delta_j),$ then we have, 
\begin{eqnarray}\label{eq:2.14}
  \|(I - Q_{n,j})f \|_{\Delta_j, \infty} &\leq& C_5 \|f^{(r+1)} \|_{\Delta_j, \infty}  h^{r+1}.
\end{eqnarray}
Thus, if $f \in C^{r+1} [0, 1],$ then
\begin{eqnarray}\label{eq:2.15}
  \|(I - Q_{n})f \|_{ \infty} &=& O \left (  h^{r+1} \right ).
\end{eqnarray}
   \subsection{Discrete Projection Methods}
   
   We define below the discrete versions of various projection methods given in Section 1 by replacing the 
   integral operator $ \mathcal{K}$ by the Nystr\"{o}m operator $\mathcal{K}_m$ and the orthogonal projection 
   $\pi_n$ by the discrete orthogonal projection $Q_n.$
 
 \noindent
Discrete Galerkin Method:
\begin{equation}\nonumber
z_n^G - Q_n \mathcal{K}_m  (z_n^G) = Q_n f.
\end{equation}
Discrete Iterated Galerkin Method:
\begin{equation}\nonumber
{z}_n^S -  \mathcal{K}_m  (Q_n {z}_n^S) =  f.
\end{equation}
The discrete modified projection operator is defined as 
\begin{equation}\label{eq:2.16}
\tilde {\mathcal{K}}_n^M (x) = Q_n \mathcal{K}_m (x) + \mathcal{K}_m (Q_n x) - Q_n \mathcal{K}_m (Q_n x).
\end{equation}
Discrete Modified Projection method:
\begin{equation}\label{eq:2.17}
{z_n^M - \tilde {\mathcal{K}}_n^M (z_n^M) = f.}
 \end{equation}

\noindent
Discrete Iterated Modified Projection method: 
\begin{equation}\label{eq:2.18} 
{\tilde{z}_n^M =  {\mathcal{K}}_m (z_n^M) + f.}
\end{equation}

\setcounter{equation}{0}

\section{Piecewise polynomial approximation : $\mathbf{ r \geq 1 }$ }

In this section we consider the case $r \geq 1$ and obtain orders of convergence in the discrete modified projection method and its iterated version. 

\subsection{Preliminary results}

In Proposition 3.1 we first obtain an error estimate for the term 
$\| \mathcal {K}_m'  (\varphi) (I - Q_n) v \|_\infty, $  where $v \in C^{r+1} [0, 1].$ Note that the term
$\mathcal {K}_m'  (\varphi) (I - Q_n) v  (s)$ needs to be treated differently depending upon whether $s$ is a partition point of the partition (\ref{eq:2.9}) or $s \in (t_{i-1}, t_i)$ for some $i.$  Using this result, we obtain an error estimate for the term
$\| \mathcal {K}_m'  (\varphi) (I - Q_n) \mathcal {K}_m'  (\varphi) (I - Q_n)v \|_\infty $ in
Proposition 3.2. These estimates are used in obtaining orders of convergence of $z_n^M$ and $\tilde{z}_n^M.$

Let
\begin{equation}\nonumber
\ell_* (s, t) =   \ell (s, t, \varphi (t)), \;\;\; 0 \leq s, t \leq 1,
\end{equation}
where $ \displaystyle {\ell (s, t, u) = \frac {\partial \kappa ( s, t, u) } { \partial u}}$. Then
$$
\ell_* (s,t) = \left\{ {\begin{array}{ll}
\ell_{1, *} (s, t) = \ell_1 (s, t, \varphi (t)), \;\;\; 0 \leq t \leq s \leq 1, \\
\ell_{2, *} (s, t) = \ell_2 (s, t, \varphi (t)), \;\;\; 0 \leq s \leq t \leq 1.
\end{array}}\right.
$$ 
Since  $\varphi \in C^{r+1} [0, 1],$ it follows that
 $$ \ell_{1, *} \in C^{r+1} ( \{ 0 \leq t \leq s \leq 1\}) \; \mbox {and} \; \ell_{2, *} \in C^{r+1} ( \{ 0 \leq s \leq t \leq 1\}).$$

\noindent
We introduce the following notation.
For a fixed $s \in [0, 1],$ define
\begin{equation}\nonumber
\ell_{*, s} (t) = \ell_* (s, t), \; t \in [0, 1].
\end{equation}
Note that
\begin{eqnarray}\nonumber
\mathcal {K}_m'  (\varphi) v (s)  & = &  \tilde {h}  \sum_{j = 1}^n \sum_{\nu =1}^p  \sum_{q=1}^\rho  w_q \;  
 \ell_* (s,  \zeta_q^{(j-1) p + \nu} ) v (\zeta_q^{(j-1) p + \nu}) = \sum_{j=1}^n \inp {\ell_{*, s}} {v}_{\Delta_j}.
\end{eqnarray}
 Let
$$ C_6 = \max_{1 \leq j \leq r+1} \left \{ \max_{\stackrel { 0 \leq t \leq  s \leq 1 }{|u| \leq \|\varphi \|_\infty  }} 
\left | D^{(0,j, 0)} \ell_{1} (s, t, u)  \right |, \max_{\stackrel { 0 \leq s \leq  t \leq 1 }{|u| \leq \|\varphi \|_\infty }} 
\left | D^{(0, j, 0)} \ell_{2}  (s, t, u)  \right |   \right \}.$$

The following proposition is crucial. It will be used several times in what follows.

\begin{proposition}\label{prop:3.1}
 If  $v \in C^{ r + 1} [0, 1],$ 
 then
\begin{eqnarray}\label{eq:3.1}
\| \mathcal {K}_m'  (\varphi) (I - Q_n) v \|_\infty 
& \leq &  (C_5)^2 C_6  \|v^{(r+1)} \|_\infty  h^{r +3}.
\end{eqnarray}
\end{proposition}

\begin{proof}
For $s \in [0, 1],$
\begin{eqnarray*}
\mathcal {K}_m'  (\varphi)  (I - Q_n ) v (s) 
& = & \sum_{j=1}^n \inp {  \ell_{*, s}} {(I - Q_{n,j}  )v}_{\Delta_j} = \sum_{j=1}^n \inp {(I - Q_{n,j}  ) \ell_{*, s}} {(I - Q_{n,j}  )v}_{\Delta_j}.
\end{eqnarray*} 
Case 1:  $ s = t_i$ for some $i \in \left\{0, 1, \ldots, n \right\}.$ Then 
$\ell_{*, s} \in C^{r+1} (\Delta_j)$ 
for $j = 1, \ldots, n.$ Since $v \in C^{r+1} [0, 1],$
it follows  from (\ref{eq:2.14}),
\begin{eqnarray}\nonumber
\max_{0 \leq i \leq n} |\mathcal{K}_m (I - Q_n ) x (t_i) | &\leq&  \sum_{j=1}^n \|(I - Q_{n,j}  ) \ell_{*, s}\|_{\Delta_j, \infty} 
\|(I - Q_{n, j}) v\|_{\Delta_j, \infty} h\\\label{eq:3.2}
& \leq &  (C_5)^2 C_6 \|v^{(r+1)}\|_\infty   h^{2 r + 2}.
\end{eqnarray}
Case 2: $s \in (t_{i-1}, t_i)$ for some $i \in \left\{1, 2, \ldots, n \right\}.$ We write
\begin{eqnarray}\nonumber
\mathcal {K}_m'  (\varphi)  (I - Q_n ) v (s) & = & 
\sum_{\stackrel {j=1} {j \neq i}}^n \inp { (I - Q_{n,j}  ) \ell_{*, s}} {(I - Q_{n,j}  )v}_{\Delta_j}\\\label{eq:3.3}
&&~ +  \inp { (I - Q_{n,i}  ) \ell_{*, s}} {(I - Q_{n,i}  )v}_{\Delta_i}.
\end{eqnarray}
For $j \neq i,$ $\ell_{*, s} \in C^{r+1} (\Delta_j)$ and $v  \in C^{r+1} (\Delta_j).$
Hence 
\begin{eqnarray}\label{eq:3.4}
\left |\sum_{\stackrel {j=1} {j \neq i}}^n \inp { (I - Q_{n,j}  ) \ell_{*, s}} {(I - Q_{n,j}  )v}_{\Delta_j} \right |
\leq (C_5)^2 C_6 \|v^{(r+1)}\|_\infty  (n-1) h^{2 r + 3}.
\end{eqnarray}
We now consider the case $j = i.$ Note that $\ell_{*, s}$ is only continuous on $[t_{i-1}, t_i].$ 
Define a constant function:
$$ g_i (t) = \ell_{*, s} (s) , \;\; t \in [t_{i-1}, t_i].$$
Note that
$\hspace*{0.5 cm}
\displaystyle {\inp { (I - Q_{n,i}  ) \ell_{*, s}} {(I - Q_{n,i}  )v}_{\Delta_i}   =  \inp { \ell_{*, s} - g_i} {(I - Q_{n,i}  )v}_{\Delta_i}.}$
For $t \in [t_{i-1}, t_i],$
\begin{eqnarray*}
 \ell_{*, s} (t) - g_i (t) 
& = & \left\{ {\begin{array}{ll}
D^{(0,1)} \ell_{1, *} (s,  \theta_t) (t - s),  \;\;\;  \theta_t \in (t,   s),   \\
D^{(0,1)} \ell_{2, *} (s, \eta_t ) (t - s),  \;\;\;    \eta_t \in (s, t).
\end{array}}\right.
\end{eqnarray*}
Thus,
\begin{eqnarray}\label{eq:3.5}
\left |\inp { (I - Q_{n,i}  ) \ell_{*, s}} {(I - Q_{n,i}  )v}_{\Delta_i} \right |
& \leq &  C_5  C_6 \|v^{(r+1)} \|_\infty h^{r+3}.
\end{eqnarray}
Without loss of generality, let $C_5 \geq 1.$ 
 From (\ref{eq:3.3}), (\ref{eq:3.4}) and (\ref{eq:3.5}) we obtain,
\begin{eqnarray}\nonumber
|\mathcal {K}_m'  (\varphi)  (I - Q_n ) v (s)| 
& \leq &  (C_5)^2 C_6   \|v^{(r+1)} \|_\infty  h^{r + 3}.
\end{eqnarray}
Combining  the above estimate with (\ref{eq:3.2}) we obtain the required result.
\end{proof}


\begin{proposition}\label{prop:3.2}
If  $v \in C^{ r + 1} [0, 1],$ 
 then 
\begin{equation}\label{eq:3.6}
\|  \mathcal {K}_m'  (\varphi) (I - Q_n )  \mathcal {K}_m'  (\varphi) (I - Q_n) v \|_\infty =
O \left ( h^{r + 5} \right ).
\end{equation}
Also,
\begin{eqnarray}\label{eq:3.7}
\left \| \mathcal {K}_m'  (\varphi) (I - Q_n ) \mathcal {K}_m'  (\varphi)   \right \| = O \left (   h^{ 2 } \right ).
\end{eqnarray}
\end{proposition}

\begin{proof}
The proof of (\ref{eq:3.6}) is similar to that of (\ref{eq:3.1}). For $s \in [0, 1],$ we write
\begin{eqnarray*}
\mathcal {K}_m'  (\varphi)  (I - Q_n ) \mathcal {K}_m'  (\varphi)  (I - Q_n ) v (s) 
& = & \sum_{j=1}^n \inp { (I - Q_{n,j}  ) \ell_{*, s} } {    \mathcal {K}_m'  (\varphi) (I - Q_n ) v}_{\Delta_j}.
\end{eqnarray*}
If $s = t_i,$ for some $i,$ then using (\ref{eq:2.14}) and (\ref{eq:3.1})
we obtain
\begin{eqnarray}\nonumber
 \left | \mathcal {K}_m'  (\varphi) (I - Q_n ) \mathcal {K}_m'  (\varphi)  (I - Q_n ) v (t_i) \right | 
 \leq  (C_5)^3 (C_6)^2   \|v^{(r+1)} \|_\infty  h^{2 r + 4}.
\end{eqnarray}
If $s \in (t_{i-1}, t_i),$ then we write
\begin{eqnarray*}
\mathcal {K}_m'  (\varphi)  (I - Q_n )  \mathcal {K}_m'  (\varphi)  (I - Q_n ) v (s) 
& = &  \sum_{\stackrel {j=1} {j \neq i}}^n \inp { (I - Q_{n,j}  ) \ell_{*, s}} {   \mathcal {K}_m'  (\varphi)  (I - Q_n )   v}_{\Delta_j}\\
 & + &  \inp {   \ell_{*, s} - g_i }  {  \mathcal {K}_m'  (\varphi)  (I - Q_n ) v}_{\Delta_i}.
\end{eqnarray*}
Proceeding as in the proof of Proposition 3.1, we obtain
\begin{eqnarray}\nonumber
|\mathcal {K}_m'  (\varphi) (I - Q_n ) \mathcal {K}_m'  (\varphi) (I - Q_n ) v (s) | 
 &\leq & (C_5)^3 (C_6)^2  \|v^{(r+1)} \|_\infty h^{ r + 5 }.
\end{eqnarray}
The estimate (\ref{eq:3.6}) follows from the above two estimates.

In order to prove  (\ref{eq:3.7}), consider $v \in C [0, 1].$ Let
 $ s = t_i$ for some $i.$ Then 
\begin{eqnarray}\label{eq:3.8}
 \left | \mathcal {K}_m'  (\varphi) (I - Q_n )  v (s) \right | 
& \leq & C_5 C_6  \|v \|_\infty  h^{ r + 1 }.
\end{eqnarray}
Now
let $s \in (t_{i-1}, t_i). $ We write
\begin{eqnarray}\nonumber
\mathcal {K}_m'  (\varphi)  (I - Q_n )    v (s) 
=  \sum_{\stackrel {j=1} {j \neq i}}^n \inp { (I - Q_{n,j}  ) \ell_{*, s}} {    v}_{\Delta_j}
 +  \inp {  (I - Q_{n,i}  ) (\ell_{*, s} - g_i) }  {v}_{\Delta_i}.
\end{eqnarray}
and obtain
\begin{eqnarray}\nonumber
|\mathcal {K}_m'  (\varphi) (I - Q_n ) v (s) | 
 &\leq &   (1 + C_4 + C_5) C_6    \|v \|_\infty h^{ 2 }.
\end{eqnarray}
Combining (\ref{eq:3.8}) and the above estimate, we obtain
\begin{eqnarray}\label{eq:3.9}
\|\mathcal {K}_m'  (\varphi) (I - Q_n ) v \|_\infty
 &\leq &   (1 + C_4 + C_5) C_6  \|v \|_\infty h^{ 2 }.
\end{eqnarray}
Since from (\ref{eq:2.3}),
$  \| \mathcal {K}_m'  (\varphi)  v\|_\infty \leq  C_1  \|v\|_\infty,$
we obtain
\begin{eqnarray}\nonumber
\|\mathcal {K}_m'  (\varphi) (I - Q_n ) \mathcal {K}_m'  (\varphi)  v \|_\infty 
 &\leq & \; C_1 (1 + C_4 + C_5) C_6  \|v \|_\infty h^{ 2} 
\end{eqnarray} and the required result follows taking the supremum over unit ball in $C[0, 1].$
\end{proof}

\subsection{Error in the  discrete modified projection method}

As in Grammont \cite{Gram1}, it can be shown that there is a $\delta_0 > 0$ such that (\ref{eq:2.17}) has a unique
solution $z_n^M$ in $\mathcal{B} (\varphi, \delta_0)$ and that 
\begin{eqnarray}\nonumber
\| z_n^M - \varphi\|_\infty 
& \leq & 6  \left \| \left (I - \mathcal{K}' (\varphi) \right )^{-1} \right \| \left ( \|\mathcal{K} (\varphi) - 
\mathcal{K}_m (\varphi)\|_\infty + \|\mathcal{K}_m (\varphi) -\tilde{\mathcal{K}}_n^M (\varphi)\|_\infty \right ).\\\label{eq:3.10}
\end{eqnarray}
In the following theorem, we obtain the order of convergence of the discrete modified projection solution.
\begin{theorem}\label{thm:3.3}
Let  $r \geq 1,$  $ \kappa$ be of class $\mathcal{G}_2 ({r+1}, 0)$
and $ f \in C^{r+1} [0, 1].$  
Let $\varphi$ be the unique solution of (\ref{eq:1.2}) and assume that $1$ is not an eigenvalue of $\mathcal{K}' (\varphi).$ 
Let $\mathcal{X}_n$ be the space of piecewise polynomials of degree $\leq r $ with respect to the partition (\ref{eq:2.9})
and $Q_n$ be the discrete orthogonal projection defined by
(\ref{eq:2.12}). Let $z_n^M $ be the discrete modified projection solution in $\mathcal{B} (\varphi, \delta_0 ).$ Then
\begin{equation}\label{eq:3.11}
\|z_n^M - \varphi \|_\infty = O (\max \{\tilde{h}^{2}, h^{r + 3}\}).
\end{equation}
\end{theorem}
\begin{proof}
From (\ref{eq:2.5}),
\begin{eqnarray}\label{eq:3.12}
\|\mathcal{K} (\varphi) -  \mathcal{K}_m (\varphi)\|_\infty = O \left ( \tilde{h}^{2} \right ).
\end{eqnarray}
Since $\varphi \in C^{r+1} [0, 1],$ it follows from (\ref{eq:2.15}) that $\displaystyle {\| Q_n \varphi - \varphi \|_\infty = O (h^{r+1}).}$
Note that
\begin{eqnarray}\nonumber
\|\mathcal{K}_m (\varphi) - \tilde{\mathcal{K}}_n^M (\varphi)  \|_\infty  
& \leq & \| (I - Q_n) (\mathcal{K}_m (Q_n\varphi) - \mathcal{K}_m (\varphi) - \mathcal {K}_m'  (\varphi) (Q_n \varphi - \varphi ) )\|_\infty\\\nonumber
&& + \| (I - Q_n) \mathcal {K}_m'  (\varphi) (Q_n \varphi - \varphi )\|_\infty. 
\end{eqnarray}
From (\ref{eq:2.7}), (\ref{eq:2.8}) and (\ref{eq:2.15}),
\begin{eqnarray}\nonumber
\| \mathcal{K}_m (Q_n\varphi) - \mathcal{K}_m (\varphi) - \mathcal {K}_m'  (\varphi) (Q_n \varphi - \varphi ) \|_\infty \leq    {C_2 }  \| Q_n \varphi - \varphi \|_\infty^2 = O (h^{2 r + 2}).
\end{eqnarray}
By (\ref{eq:2.13}) and Proposition \ref{prop:3.1},
\begin{eqnarray}\nonumber
\| (I - Q_n) \mathcal {K}_m'  (\varphi) (Q_n \varphi - \varphi )\|_\infty 
\leq (1 + C_4) \| \mathcal {K}_m'  (\varphi) (Q_n \varphi - \varphi )\|_\infty  = O (h^{r + 3}).
\end{eqnarray}
Since $r \geq 1,$ it then follows that
\begin{eqnarray}\nonumber
\|\mathcal{K}_m (\varphi) - \tilde{\mathcal{K}}_n^M (\varphi)  \|_\infty = O (h^{r + 3}).
\end{eqnarray}
The required result follows from (\ref{eq:3.10}), (\ref{eq:3.12}) and the above estimate.
\end{proof}

\begin{remark}
It can be  shown that
\begin{equation}\label{*}
 \|z_n^G - \varphi \|_\infty = O \left (\max \left \{\tilde{h}^{2}, h^{r + 1 } \right \} \right ), 
 \;\;\; \|z_n^S - \varphi \|_\infty = O \left (\max \left \{\tilde{h}^{2}, h^{r +3} \right \} \right ).
\end{equation}
Thus the order of convergence of $z_n^S$ and $z_n^M$ is the same. We prove  the estimate (\ref{eq:3.11}) as 
it is needed for obtaining the order of convergence in the iterated discrete modified projection method. 
\end{remark}

\subsection{Error in the  discrete iterated modified projection method} 

Note that
\begin{equation}\nonumber
\tilde{z}_n^M - \varphi_m =  \mathcal {K}_m ({z}_n^M)  - \mathcal {K}_m (\varphi_m).
\end{equation}
From (\ref{eq:2.7}) and (\ref{eq:2.8}), 
\begin{eqnarray}\nonumber
\mathcal {K}_m ({z}_n^M) - \mathcal {K}_m (\varphi_m)  = \mathcal {K}_m' (\varphi_m) (z_n^M - \varphi_m)  + 
O \left ( \|z_n^M - \varphi_m\|_\infty^2 \right ).
\end{eqnarray}
From (\ref{eq:2.6}) and Theorem \ref{thm:3.3}, we obtain
\begin{equation}\nonumber
\|z_n^M - \varphi_m \|_\infty \leq  \|z_n^M - \varphi \|_\infty + \|\varphi - \varphi_m \|_\infty = O (\max \{\tilde{h}^{2}, 
h^{r + 3}\}).
\end{equation}
Thus,
\begin{eqnarray}\label{eq:3.13}
\tilde{z}_n^M - \varphi_m =  \mathcal {K}_m' (\varphi_m) (z_n^M - \varphi_m) + O (\max \{\tilde{h}^{2}, 
h^{r + 3}\}^2).
\end{eqnarray}
We quote the following result from Kulkarni-Rakshit \cite{Kul1}:
\begin{eqnarray}\nonumber
&&\mathcal{K}_m'   (\varphi_m) ( z_n^M - \varphi_m ) \\\nonumber
& = &  
 - \left [ I - {\mathcal{K}}_m'   (\varphi_m) \right]^{-1}  \mathcal{K}_m'   (\varphi_m) \left \{  \mathcal {K}_m (\varphi_m) -  \tilde{\mathcal{K}}_n^M (\varphi_m) \right \}\\\nonumber
&&+ \left [ I - {\mathcal{K}}_m'   (\varphi_m) \right]^{-1} \mathcal{K}_m'   (\varphi_m) \left \{
\tilde{\mathcal{K}}_n^M (z_n^M) - \tilde{\mathcal{K}}_n^M (\varphi_m)  - \left (  \tilde{\mathcal{K}}_n^M \right )' (\varphi_m) 
(z_n^M - \varphi_m) \right \}\\\label{eq:3.14}
&& + \left [ I - {\mathcal{K}}_m'   (\varphi_m) \right]^{-1} \mathcal{K}_m'   (\varphi_m) \left \{
\left (\left (  \tilde{\mathcal{K}}_n^M \right )' (\varphi_m)  -  \mathcal{K}_m'   (\varphi_m) \right  ) (z_n^M - \varphi_m) \right \}.
\end{eqnarray}
We obtain below orders of convergence for the three terms in (\ref{eq:3.14}).


\begin{proposition}\label{prop:3.4}
Let $\varphi_m$ be the Nystr\"{o}m solution. Then
\begin{eqnarray}\nonumber
 \left \| \mathcal {K}_m'  (\varphi_m) \left (\mathcal {K}_m (\varphi_m) -  \tilde{\mathcal{K}}_n^M (\varphi_m) \right ) \right \|_\infty = O \left ( h^{4} \max \{\tilde{h}^{2}, h^{r +1}\} \right ).
 \end{eqnarray}
\end{proposition}
\begin{proof}
Let $v \in C[0, 1].$ Then from (\ref{eq:2.4}),  (\ref{eq:2.6}), (\ref{eq:2.13}) and (\ref{eq:3.9}), 
\begin{eqnarray}\nonumber
\| \mathcal{K}_m' (\varphi_m)(I - Q_n) v \|_\infty &\leq& \| \left [ \mathcal{K}_m' (\varphi_m) -  \mathcal{K}_m' (\varphi) \right ]
(I - Q_n) v \|_\infty 
+  \| \mathcal{K}_m' (\varphi)(I - Q_n) v \|_\infty \\\nonumber
&\leq & C_2  (1 + C_4) \|v\|_\infty \|\varphi_m - \varphi \|_\infty  
+  (1 + C_4 + C_5) C_6  \|v\|_\infty h^{2}\\\label{eq:3.15}
&\leq & C_7 \|v\|_\infty  h^{2}.
\end{eqnarray}
Note that
\begin{eqnarray}\nonumber
\mathcal{K}_m (\varphi_m) - \tilde{\mathcal{K}}_n^M (\varphi_m)  
& = & - (I - Q_n) (\mathcal{K}_m (Q_n\varphi_m) - \mathcal{K}_m (\varphi_m) - \mathcal {K}_m'  (\varphi_m) (Q_n \varphi_m - \varphi_m ) )\\\label{eq:3.16}
&& -  (I - Q_n) \mathcal {K}_m'  (\varphi_m) (Q_n \varphi_m - \varphi_m ). 
\end{eqnarray}
Let
\begin{eqnarray}\nonumber
y_n = \mathcal{K}_m (Q_n\varphi_m) - \mathcal{K}_m (\varphi_m) - \mathcal {K}_m'  (\varphi_m) (Q_n \varphi_m - \varphi_m ) 
= R (Q_n \varphi_m - \varphi_m ) .
\end{eqnarray}
Then by (\ref{eq:3.15})
\begin{eqnarray}\nonumber
\| \mathcal{K}_m' (\varphi_m)(I - Q_n) y_n \|_\infty &\leq& C_7 \|y_n\|_\infty  h^{2}.
\end{eqnarray}
By (\ref{eq:2.8})
\begin{eqnarray}\nonumber
\|y_n \|_{\infty} = \|R (Q_n \varphi_m - \varphi_m )\|_{ \infty } \leq 
 C_2  \|Q_n \varphi_m - \varphi_m \|_\infty^2.
\end{eqnarray}
Using (\ref{eq:2.6}),  (\ref{eq:2.13}) and (\ref{eq:2.15}),   we obtain
\begin{eqnarray}\label{eq:3.17}
\|Q_n \varphi_m - \varphi_m \|_\infty 
= O (\max \{\tilde{h}^{2}, h^{r+1 }\}).
\end{eqnarray}
Thus,
\begin{eqnarray}\label{eq:3.18}
\| \mathcal{K}_m' (\varphi_m)(I - Q_n) y_n \|_\infty = O \left (  h^{2}  \max \{\tilde{h}^{2}, h^{r +1}\}^2 \right ).
\end{eqnarray}
Using (\ref{eq:3.6}) it can be checked
  that
 \begin{eqnarray*}\nonumber
&&\left \|\mathcal {K}_m'  (\varphi_m)   (I - Q_n) \mathcal {K}_m'  (\varphi_m)   (I - Q_n) \varphi_m \right \|_{\infty}
 = O \left ( h^{4} \max \{\tilde{h}^{2}, h^{r +1}\} \right ).
\end{eqnarray*}
The required result then follows from (\ref{eq:3.16}),  (\ref{eq:3.18}) and the above estimate.
\end{proof}

\begin{proposition}\label{prop:3.5}
Let $\varphi_m$ be the Nystr\"{o}m solution and $z_n^M$ be the discrete modified projection solution. Then
\begin{eqnarray}\nonumber
&&\left \|\tilde {\mathcal{K}}_n^M (z_n^M) - \tilde{\mathcal{K}}_n^M (\varphi_m)  - \left (  \tilde{\mathcal{K}}_n^M \right )' (\varphi_m) 
(z_n^M - \varphi_m) \right \|_\infty = O \left (\max \left \{ \tilde{h}^{2},  h^{ r + 3} \right \}^2 \right).
\end{eqnarray}
\end{proposition}
\begin{proof}
Note that  for $m$ and $n$ big enough, 
 $\displaystyle {\varphi_m, z_n^M \in \mathcal {B} \left (\varphi, \delta_0 \right ). }$
By the generalized Taylor's theorem,
\begin{eqnarray*}
&&\tilde {\mathcal{K}}_n^M (z_n^M) (s) - \tilde{\mathcal{K}}_n^M (\varphi_m)  (s) - \left (  \tilde{\mathcal{K}}_n^M \right )' (\varphi_m) 
(z_n^M - \varphi_m) (s) \\
&& \hspace*{0.5 in} = 
\int_0^1 ( 1 - \theta) \left (  \tilde{\mathcal{K}}_n^M \right )'' \left (\varphi_m + \theta (z_n^M - \varphi_m) \right ) (z_n^M - \varphi_m)^2
(s) \; d \theta.
\end{eqnarray*}
Hence
\begin{eqnarray}\nonumber
&&\left \|\tilde {\mathcal{K}}_n^M (z_n^M) - \tilde{\mathcal{K}}_n^M (\varphi_m)  - \left (  \tilde{\mathcal{K}}_n^M \right )' (\varphi_m) 
(z_n^M - \varphi_m) \right \|_\infty  \\\nonumber
&& \hspace*{0.5 in} \leq \frac {1} {2} \max_{0 \leq \theta \leq 1}  
\left \|\left (\tilde {\mathcal{K}}_n^M \right )'' \left  (\varphi_m + \theta (z_n^M - \varphi_m) \right ) \right \| \| z_n^M - \varphi_m\|_\infty^2.
\end{eqnarray}
It can be shown that
\begin{eqnarray}\nonumber
\max_{0 \leq \theta \leq 1}  
\left \|\left (\tilde {\mathcal{K}}_n^M \right )'' \left  (\varphi_m + \theta (z_n^M - \varphi_m) \right ) \right \|
\leq  C_8.
\end{eqnarray}
We skip the details.
The required result then follows from Theorem 3.3. 
\end{proof}


\begin{proposition}\label{prop:3.6}
Let $\varphi_m$ be the Nystr\"{o}m solution and $z_n^M$ be the discrete modified projection solution. Then
\begin{eqnarray}\nonumber
&&\left \|\mathcal{K}_m'   (\varphi_m) 
\left (\left (  \tilde{\mathcal{K}}_n^M \right )' (\varphi_m)  -  \mathcal{K}_m'   (\varphi_m) \right  ) (z_n^M - \varphi_m) 
\right \|_\infty
= O \left ( h^2 \max \left \{ \tilde{h}^{2},  h^{r + 3} \right \} \right).
\end{eqnarray}
\end{proposition}

\begin{proof}
Note that
\begin{eqnarray*}
\mathcal{K}_m'   (\varphi_m) \left (\left (  \tilde{\mathcal{K}}_n^M \right )' (\varphi_m)  -  \mathcal{K}_m'   (\varphi_m) \right  ) 
& = &\mathcal{K}_m'   (\varphi_m) (I - Q_n) (\mathcal{K}_m'  (Q_n \varphi_m )  -  \mathcal{K}_m'   (\varphi_m)) Q_n \\
& - & \mathcal{K}_m'   (\varphi_m)  (I - Q_n) \mathcal{K}_m'   (\varphi_m) (I - Q_n).
\end{eqnarray*}
Using (\ref{eq:2.3}) and (\ref{eq:3.15}) it can be shown that
$$\| \mathcal{K}_m'   (\varphi_m)  (I - Q_n) \mathcal{K}_m'   (\varphi_m) \| = O (h^{ 2}).$$
By (\ref{eq:2.6}) and (\ref{eq:3.17}),
\begin{eqnarray*}
\|\mathcal{K}_m'  (Q_n \varphi_m )  -  \mathcal{K}_m'   (\varphi_m)\|\leq  C_2 \|Q_n \varphi_m - \varphi_m \|_\infty = 
O (\max \{\tilde{h}^2, h^{r+1} \}).
\end{eqnarray*}
Since
by (\ref{eq:2.3}), 
$\displaystyle { \|\mathcal{K}_m'   (\varphi_m) \| \leq C_1},$
it follows that
\begin{eqnarray}\nonumber
\left \|\mathcal{K}_m'   (\varphi_m) \left (\left (  \tilde{\mathcal{K}}_n^M \right )' (\varphi_m)  -  \mathcal{K}_m'   (\varphi_m) \right  ) \right \|= O (h^2).
\end{eqnarray}
The required result follows using the estimate for $\|z_n^M - \varphi\|_\infty $ from  Theorem 3.3.
\end{proof}

We now prove our main result about 
 the order of convergence in the discrete iterated modified projection method.


\begin{theorem}\label{thm:3.7}
Let   $r \geq 1,$ $ \kappa$ be of class $\mathcal{G}_2 ({r+1}, 0)$
and $ f \in C^{r+1} [0, 1].$    
Let $\varphi$ be the unique solution of (\ref{eq:1.2}) and assume that $1$ is not an eigenvalue of $\mathcal{K}' (\varphi).$ 
Let $\mathcal{X}_n$ be the space of piecewise polynomials of degree $\leq r $ with respect to the partition 
(\ref{eq:2.9})
and $Q_n$ be the discrete orthogonal projection defined by
(\ref{eq:2.12}). Let $\tilde{z}_n^M$ be the discrete iterated modified projection solution defined by (\ref{eq:2.18}).
 Then
\begin{eqnarray}\label{eq:3.19}
\|\tilde{z}_n^M - \varphi \|_\infty = O \left ( \max \left \{\tilde{h}^2,   
h^{r + 5} \right  \}\right).
\end{eqnarray}
\end{theorem}
\begin{proof}
We have from (\ref{eq:3.13}) 
\begin{eqnarray}\nonumber
\tilde{z}_n^M - \varphi_m =  \mathcal {K}_m' (\varphi_m) (z_n^M - \varphi_m) + O (\max \{\tilde{h}^{2}, 
h^{r + 3}\}^2).
\end{eqnarray}
From (\ref{eq:3.14}) recall that 
\begin{eqnarray}\nonumber
&&\mathcal{K}_m'   (\varphi_m) ( z_n^M - \varphi_m ) \\\nonumber
& = &  
 - \left [ I - {\mathcal{K}}_m'   (\varphi_m) \right]^{-1}  \mathcal{K}_m'   (\varphi_m) \left \{  \mathcal {K}_m (\varphi_m) -  \tilde{\mathcal{K}}_n^M (\varphi_m) \right \}\\\nonumber
&&+ \left [ I - {\mathcal{K}}_m'   (\varphi_m) \right]^{-1} \mathcal{K}_m'   (\varphi_m) \left \{
\tilde{\mathcal{K}}_n^M (z_n^M) - \tilde{\mathcal{K}}_n^M (\varphi_m)  - \left (  \tilde{\mathcal{K}}_n^M \right )' (\varphi_m) 
(z_n^M - \varphi_m) \right \}\\\nonumber
&& + \left [ I - {\mathcal{K}}_m'   (\varphi_m) \right]^{-1} \mathcal{K}_m'   (\varphi_m) \left \{
\left (\left (  \tilde{\mathcal{K}}_n^M \right )' (\varphi_m)  -  \mathcal{K}_m'   (\varphi_m) \right  ) (z_n^M - \varphi_m) \right \}.
\end{eqnarray}
By Proposition 4.2 from Kulkarni-Rakshit \cite{Kul1}, we have
$$ \left \|\left [ I - {\mathcal{K}}_m'   (\varphi_m) \right]^{-1} \right \| \leq 
4 \left \| \left ( I - \mathcal{K}' (\varphi ) \right )^{-1} \right \|.$$
Hence by Proposition \ref{prop:3.4},
  Proposition \ref{prop:3.5},
 and Proposition \ref{prop:3.6},
 \begin{eqnarray}\nonumber
 &&\left \| \mathcal{K}_m'   (\varphi_m) ( z_n^M - \varphi_m )  \right \|_\infty 
 = O \left (  \max \left \{  h^4 \max \{\tilde{h}^{2}, h^{r +1}\},    h^{2} \max \{\tilde{h}^{2}, h^{r +3}\}\right \} \right ),
 \end{eqnarray}
 It follows that
  \begin{eqnarray}\nonumber
 &&\left \| \tilde{z}_n^M - \varphi_m  \right \|_\infty =
 O \left (     h^{2} \max \{\tilde{h}^{2}, h^{r +3}\} \right ).
 \end{eqnarray}
 Since, $\tilde{z}_n^M - \varphi = \tilde{z}_n^M - \varphi_m + \varphi_m - \varphi $ and $ \left \| \varphi - \varphi_m \right \|_\infty = O \left ( \tilde{h}^{2} \right ),$
 the required result  follows. 
\end{proof}

\setcounter{equation}{0}

\section{Piecewise constant polynomial approximation : $\mathbf{ r = 0 }$}

In this section we assume that $ \kappa$ is of class $\mathcal{G}_2 ({2}, 0).$
If we follow the development in Section 3, then we obtain the following orders of convergence:
 \begin{equation}\nonumber
  \| z_n^M - \varphi \|_\infty = O (h^2 ), \;\;\; 
\| \tilde{z}_n^M  - \varphi \|_\infty = O ( \max \{\tilde{h}^2, h^3 \} ).
\end{equation} 
But by looking at the proofs more carefully, we are able to improve the above estimates. 
More specifically, while for $r \geq 1,$  if $v \in C^{r+1} [0, 1],$ then both $\| \mathcal {K}_m'  (\varphi) (I - Q_n) v \|_\infty $
and $\|(I- Q_n) \mathcal {K}_m'  (\varphi) (I - Q_n) v \|_\infty $ are of the same order, we could show that if $r=0,$ then
$$ \| \mathcal {K}_m'  (\varphi) (I - Q_n) v \|_\infty = O (h^2) \;\;\; \mbox{and} \;\;\;
\| (I - Q_n) \mathcal {K}_m'  (\varphi) (I - Q_n) v \|_\infty = O (h^3).$$ 
This is the essential point in proving the estimates (\ref{eq:1.5}).

Consider $\mathcal{X}_n$ to be the space of piecewise constant functions with respect to the partition (\ref{eq:2.9}).  Thus, $r = 0.$ 
We
choose Gauss 2 point rule as a basic quadrature rule:
  \begin{equation}\nonumber
  \int_0^1 f (t) d t \approx  w_1 f (\mu_1) + w_2 f (\mu_2),
  \end{equation}
  where
  \begin{eqnarray*}
  w_1 = w_2 = \frac {1} {2}, \; \mu_1 = \frac {1} {2} - \frac {1} {2 \sqrt{3}}, \; 
  \mu_2 = \frac {1} {2} + \frac {1} {2 \sqrt{3}}.
  \end{eqnarray*}
A composite integration rule with respect to the fine partition  (\ref{eq:2.1}) 
is then defined as
\begin{eqnarray}\nonumber
 \int_0^1 f (t) d t &\approx&  {\tilde h}  \sum_{i =1}^m   \sum_{q =1}^{2} w_q  \; f (\zeta_q^i ),   \;\;\; 
  \;\;\;  \zeta_q^i = s_{i -1} + \mu_q \tilde{h}.
\end{eqnarray}
Since $m = n p,$ the above
rule can be written as
\begin{equation}\nonumber
  \int_{0}^1 f (t) d t \approx
  \tilde {h} \ \sum_{j=1}^n \sum_{\nu = 1}^p \sum_{q = 1}^{2} w_q  f \left (\zeta_q^{(j-1) p + \nu} \right ).
  \end{equation}
The Nystr\"{o}m operator is defined as
\begin{eqnarray}\label{eq:4.1}
\mathcal{K}_m (x) (s)  & = & \tilde {h} \sum_{j=1}^n \sum_{\nu = 1}^p \sum_{q = 1}^{2}   w_q ~\kappa \left (s, \zeta_q^{(j-1) p + \nu},  x{\left ( \zeta_q^{(j-1) p + \nu} \right )} \right ).
\end{eqnarray}
Recall from (\ref{eq:2.10})  that 
for $f, g \in C (\Delta_j),$ 
\begin{equation}\nonumber
\inp {f} {g}_{\Delta_j} = \tilde {h} \sum_{\nu = 1}^p \sum_{q = 1}^{2}   w_q ~ f{\left (\zeta_q^{(j-1) p + \nu}\right )} ~ g{\left (\zeta_q^{(j-1) p + \nu} \right)}.
\end{equation}
The discrete orthogonal projection $Q_{n,j}: C (\Delta_j) \rightarrow \mathcal{P}_{0, \Delta_j}$ is defined 
as follows:
\begin{eqnarray}\nonumber
(Q_{n, j} v) (t)  =
 \frac {1} {p }  \left [ \sum_{\nu = 1}^p \sum_{q = 1}^2   
   w_q v \left (\zeta_q^{(j-1) p + \nu} \right )  \right ], \;\;\; t \in (t_{j-1}, t_j], \;\;\; 
\end{eqnarray}
and 
\begin{eqnarray}\nonumber
(Q_{n, 1} v) (0) 
&=& \frac {1} {p }  \left [ \sum_{\nu = 1}^p \sum_{q = 1}^2   
   w_q v \left (\zeta_q^{ \nu} \right )  \right ].
\end{eqnarray}
A discrete orthogonal projection $Q_n:  C[0, 1] \rightarrow \mathcal{X}_n$ is defined as 
\begin{eqnarray}\label{eq:4.2}
{Q_n v = \sum_{j=1}^n Q_{n, j} v.}
\end{eqnarray}

The following result is crucial in obtaining improved orders of convergence in the discrete modified 
projection method and its iterated version.


\begin{proposition}

If $v \in C^1 [0, 1],$ then
\begin{eqnarray}\label{eq:4.3}
\|(I - Q_{n} ) \mathcal{K}_m'(\varphi) (I - Q_n ) v \|_\infty = O (h^3),
\end{eqnarray}
\begin{eqnarray}\label{eq:4.4}
\left \| \mathcal{K}_m'(\varphi) (I - Q_n ) \mathcal{K}_m'(\varphi) (I - Q_n ) v \right \|_\infty = O (h^4). 
\end{eqnarray}
\end{proposition}
\begin{proof}
Note that
\begin{eqnarray*}
\| (I - Q_n ) \mathcal{K}_m'(\varphi) (I - Q_n ) v \|_\infty 
& = & \max _{1 \leq i \leq n} \max_{s \in [t_{i-1}, t_i]}
\left |(I - Q_{n, i} ) \mathcal{K}_m'(\varphi) (I - Q_n ) v (s) \right|.
\end{eqnarray*}
   For $s \in [t_{i-1}, t_i],$
  \begin{eqnarray} \nonumber
&&(I - Q_{n, i} ) \mathcal{K}_m'(\varphi) (I - Q_n ) v (s) \\\nonumber
   && \hspace*{1 cm} = \frac {1} {p } \sum_{\nu = 1}^p \sum_{q = 1}^2   
   w_q \left \{ \mathcal{K}_m'(\varphi) (I - Q_n ) v (s) - \mathcal{K}_m'(\varphi) (I - Q_n ) v \left (\zeta_q^{(i-1) p + \nu} \right ) \right \}  
   \\\nonumber
   && \hspace*{1 cm} = \frac {1} {p } \sum_{\nu = 1}^p \sum_{q = 1}^2   \sum_{ \stackrel {j =1} {j \neq i}}^n
   w_q  \inp {\ell_{*,s} - \ell_{*,\zeta_q^{(i-1) p + \nu}} } {(I - Q_{n, j} ) v}_{\Delta_j}\\\label{eq:4.5}
   && \hspace*{1 cm} + \frac {1} {p } \sum_{\nu = 1}^p \sum_{q = 1}^2  
   w_q  \inp {\ell_{*,s} - \ell_{*,\zeta_q^{(i -1) p + \nu}} } {(I - Q_{n, i} ) v}_{\Delta_i}.
   \end{eqnarray}
   For $j \neq i,$ 
   \begin{eqnarray*}
    \inp {\ell_{*,s} - \ell_{*,\zeta_q^{(i-1) p + \nu}} } {(I - Q_{n, j} ) v}_{\Delta_j} 
     = (s - \zeta_q^{(i-1) p + \nu}) \inp {D^{(1, 0)} \ell_{*}  (\eta_q^{{(i-1) p + \nu}}, \cdot) } {(I - Q_{n, j} ) v}_{\Delta_j},
   \end{eqnarray*}
   for some $\eta_q^{{(i-1) p + \nu}} \in (t_{i - 1}, t_i).$ 
    Define the following constant function
   $$ g_q^{{(i-1) p + \nu}} (t) =  D^{(1, 0)} \ell_{*}  {\left (\eta_q^{{(i-1) p + \nu}}, \frac {t_{j-1} + t_j} {2} \right )} , \; \; \; 
   t \in [t_{j - 1}, t_j].$$
   Then
   \begin{eqnarray*}
    && \inp {\ell_{*,s} - \ell_{*,\zeta_q^{(i-1) p + \nu}} } {(I - Q_{n, j} ) v}_{\Delta_j} \\
    &  & \hspace*{1 cm} = \left (s - \zeta_q^{(i-1) p + \nu}\right )  \inp { D^{(1, 0)} \ell_{*} (\eta_q^{(i-1) p + \nu}, \cdot) - g_q^{(i-1) p + \nu } } {(I - Q_{n, j} ) v}_{\Delta_j}.
   \end{eqnarray*}
   From (\ref{eq:2.11}) and (\ref{eq:2.14}),
   \begin{eqnarray*}
  &&  \left | \inp {\ell_{*,s} - \ell_{*,\zeta_q^{(i-1) p + \nu}} } {(I - Q_{n, j} ) v}_{\Delta_j} \right |
   \leq   C_5 \left ( \max_{s \neq t} \left |D^{(1, 1)} \ell_{*} (s, t)
   \right | \right ) \|v'\|_{ \infty} h^4.
  \end{eqnarray*}
    On the other hand,  from (\ref{eq:3.5}) with $r=0,$
    \begin{eqnarray*}
    \left | \inp {\ell_{*,s} - \ell_{*,\zeta_q^{(i-1) p + \nu}} } {(I - Q_{n, i} ) v}_{\Delta_i} \right |
    & \leq &  2  C_5 C_6  \|v'\|_\infty h^3, \;\;\; \nu = 1, \ldots, p.
    \end{eqnarray*}
   Thus, from (\ref{eq:4.5}) and the above two estimates,
\begin{eqnarray}\nonumber
 \|(I - Q_{n, i} ) \mathcal{K}_m'(\varphi) (I - Q_n ) v \|_{\Delta_i, \infty} 
& \leq &  C_5 \max\left\{ 2 C_6, \max_{s \neq t} \left |D^{(1, 1)} \ell_{*} (s, t)
\right |  \right\}   
   \|v'\|_\infty h^3.
\end{eqnarray}
  This completes the proof of (\ref{eq:4.3}).
  
  In order to prove (\ref{eq:4.4}), as before we consider two cases.
   If $s = t_i$ for some $i,$ then 
  \begin{eqnarray*}
\mathcal{K}_m'(\varphi) (I - Q_n ) \mathcal{K}_m'(\varphi) (I - Q_n ) v (s) 
& = & \sum_{j=1}^n \inp { (I - Q_{n,j}  ) \ell_{*,s} } {  (I - Q_{n,j}  )  \mathcal{K}_m'(\varphi) (I - Q_n )v}_{\Delta_j}.
\end{eqnarray*}
  If $s \in (t_{i-1}, t_i),$ then we write
  \begin{eqnarray}\nonumber
\mathcal{K}_m'(\varphi) (I - Q_n ) \mathcal{K}_m'(\varphi) (I - Q_n ) v (s) 
& = & \sum_{\stackrel {j=1} {j \neq i}}^n \inp { (I - Q_{n,j}  ) \ell_{*,s} } {  (I - Q_{n,j}  )  \mathcal{K}_m'(\varphi) (I - Q_n )v}_{\Delta_j}\\\nonumber
& + & \inp {   \ell_{*,s} - g_i} {  (I - Q_{n,i}  )\mathcal{K}_m'(\varphi) (I - Q_n )v}_{\Delta_i}.
\end{eqnarray}
Proceeding as in the proof of Proposition 3.2 and using the estimate (\ref{eq:4.3}), we obtain the required result.
  \end{proof}

\begin{theorem}\label{thm:4.2}
Let   $ \kappa$ be of class $\mathcal{G}_2 (2, 0)$
and $ f \in C^{2} [0, 1].$  
Let $\varphi$ be the unique solution of (\ref{eq:1.2}) and assume that $1$ is not an eigenvalue of $\mathcal{K}' (\varphi).$ 
Let $\mathcal{X}_n$ be the space of piecewise constant functions  with respect to the partition (\ref{eq:2.9})
and $Q_n: L^\infty [0, 1] \rightarrow \mathcal {X}_n$ be the discrete orthogonal projection defined by (\ref{eq:4.2}).
 Let $z_n^M $ be the discrete modified projection solution in $\mathcal{B} (\varphi, \delta_0 ).$ Then
\begin{equation}\label{eq:4.6}
\|z_n^M - \varphi \|_\infty = O (\max \{\tilde{h}^2, h^3 \}).
\end{equation}
\end{theorem}

\begin{proof}
Recall from (\ref{eq:3.10}) that 
\begin{eqnarray}\nonumber
\| z_n^M - \varphi\|_\infty 
& \leq & 6  \left \| \left (I - \mathcal{K}' (\varphi) \right )^{-1} \right \| \left ( \|\mathcal{K} (\varphi) - 
\mathcal{K}_m (\varphi)\|_\infty + \|\mathcal{K}_m (\varphi) -\tilde{\mathcal{K}}_n^M (\varphi)\|_\infty \right ).
\end{eqnarray}
From (\ref{eq:2.5}) we have
\begin{eqnarray}\label{eq:4.7}
\|\mathcal{K} (\varphi) - \mathcal{K}_m (\varphi)\|_\infty = O (\tilde{h}^2).
\end{eqnarray}
On the other hand,
\begin{eqnarray}\nonumber
\|\mathcal{K}_m (\varphi) - \tilde{\mathcal{K}}_n^M (\varphi)  \|_\infty  
& \leq & \| (I - Q_n) (\mathcal{K}_m (Q_n\varphi) - \mathcal{K}_m (\varphi) - \mathcal {K}_m'  (\varphi) (Q_n \varphi - \varphi ) )\|_\infty\\\label{new1}
&& + \| (I - Q_n) \mathcal {K}_m'  (\varphi) (Q_n \varphi - \varphi )\|_\infty. 
\end{eqnarray}
Recall from (\ref{eq:2.7}) that
\begin{eqnarray}\nonumber
 \mathcal{K}_m (Q_n\varphi) - \mathcal{K}_m (\varphi) - \mathcal {K}_m'  (\varphi) (Q_n \varphi - \varphi )  = 
R  ( Q_n \varphi - \varphi ),
\end{eqnarray}
where
\begin{align}\nonumber
R  ( Q_n \varphi - \varphi )(s) = \int_{0}^{1} \mathcal{K}_m''(\varphi + \theta(Q_n \varphi - \varphi )) (Q_n \varphi - \varphi )^2(s) (1-\theta) d\theta.
\end{align}
Note that
\begin{align*}
&\mathcal {K}_m''  (\varphi + \theta(Q_n \varphi - \varphi )) (Q_n \varphi - \varphi )^2 (s) \\ &\hspace{1cm} =  \tilde {h}  \sum_{j=1}^m  \sum_{q=1}^\rho  w_q \;  \frac {\partial^2 \kappa } {\partial u^2} (s,  \zeta_q^j, (\varphi + \theta(Q_n \varphi - \varphi )) (\zeta_q^j) (Q_{n, j} \varphi - \varphi )^2  (\zeta_q^j)
\end{align*}
Define
\begin{align*}
\sigma_n (s,t) = \frac {\partial^2 \kappa } {\partial u^2} (s,  t, (\varphi(t) + \theta(Q_n \varphi - \varphi )(t))
\end{align*}
and for a fixed $s \in [0, 1],$ let
\begin{equation*}
\sigma_{n, s} (t) = \sigma_n (s,t), \; t \in [0, 1].
\end{equation*}
Let
\begin{eqnarray*}
C_9 = \max\left\{ \sup_{ \stackrel {0\leq t<s\leq 1} {|u|\leq \|\varphi\|_\infty + \delta_0}} 
\left |D^{(0,1,1) } \ell_1 (s, t, u) \right | , \sup_{ \stackrel {0\leq s<t\leq 1} {|u|\leq \|\varphi\|_\infty + \delta_0}}
\left |D^{(0,1,1) } \ell_2 (s, t, u) \right | \right\},
\end{eqnarray*}
Note that
\begin{eqnarray*}
\mathcal {K}_m''  (\varphi + \theta(Q_n \varphi - \varphi )) (Q_n \varphi - \varphi )^2 (s) 
&  = &\sum_{j = 1}^n \left< (I-Q_{n, j})\sigma_{n, s}, ((I- Q_{n, j}) \varphi )^2 \right>_{\Delta_j}.
\end {eqnarray*}
If $s = t_i$ for some $i,$ then for all $j$ and if $s \in (t_{i-1}, t_i)$ for some $i,$ then for $j \neq i,$
\begin{eqnarray*}
\|(I-Q_{n, j})\sigma_{n, s}\|_{\Delta_j, \infty} \leq C_5 C_9 h.
\end{eqnarray*}
We then obtain $\;\;\; \left \|\mathcal {K}_m''  (\varphi + \theta(Q_n \varphi - \varphi )) (Q_n \varphi - \varphi )^2 \right \|_\infty = 
O (h^3).$
It follows that
\begin{eqnarray}\label{eq:4.8}
\| (I - Q_n) (\mathcal{K}_m (Q_n\varphi) - \mathcal{K}_m (\varphi) - \mathcal {K}_m'  (\varphi) (Q_n \varphi - \varphi ) \|_\infty 
& = & O (h^3).
\end{eqnarray}
Using the estimate (\ref{eq:4.3}) of Proposition 4.1 and \eqref{new1}, we thus obtain
\begin{eqnarray*}
\|\mathcal{K}_m (\varphi) - \tilde{\mathcal{K}}_n^M (\varphi)  \|_\infty   = O (h^3).
\end{eqnarray*}
The required result follows from (\ref{eq:4.7}) and the above estimate.
\end{proof}


\begin{theorem}\label{thm:4.3}
Let   $ \kappa$ be of class $\mathcal{G}_2 (2, 0)$
and $ f \in C^{2} [0, 1].$  
Let $\varphi$ be the unique solution of (\ref{eq:1.2}) and assume that $1$ is not an eigenvalue of $\mathcal{K}' (\varphi).$ 
Let $\mathcal{X}_n$ be the space of piecewise constant functions with respect to the partition (\ref{eq:2.9})
and $Q_n: L^\infty [0, 1] \rightarrow \mathcal {X}_n$ be the discrete orthogonal projection defined by
(\ref{eq:4.2}). Let $\tilde{z}_n^M$ be the discrete iterated modified projection solution defined by (\ref{eq:2.18}). Then
\begin{eqnarray}\label{eq:4.9}
\|\tilde{z}_n^M - \varphi \|_\infty = O \left ( \max \left \{\tilde{h}^2, h^{4} \right \} \ \right).
\end{eqnarray}
\end{theorem}
\begin{proof}
Recall from Section 3.3 that
\begin{eqnarray}\nonumber
\tilde{z}_n^M - \varphi_m &=& \mathcal {K}_m' (\varphi_m) (z_n^M - \varphi_m) + O (\|z_n^M - \varphi \|_\infty^2)
\end{eqnarray}
Hence by Theorem 4.2,
\begin{eqnarray}\label{eq:4.10}
\tilde{z}_n^M - \varphi_m = \mathcal {K}_m' (\varphi_m) (z_n^M - \varphi_m) +
O \left (\max \{\tilde{h}^{2}, h^{ 3} \}^2 \right )
\end{eqnarray}
We now obtain estimates for the three terms in the expression for $\mathcal {K}_m' (\varphi_m) (z_n^M - \varphi_m)$ given in (\ref{eq:3.14}). 
Note that
\begin{eqnarray}\label{eq:4.11}
\|\mathcal{K}_m (\varphi_m) - \tilde{\mathcal{K}}_n^M (\varphi_m)  \|_\infty  
& \leq & \| (I - Q_n) (\mathcal{K}_m (Q_n\varphi_m) - \mathcal{K}_m (\varphi_m) - \mathcal {K}_m'  (\varphi_m) (Q_n \varphi_m - \varphi_m ) )\|_\infty \nonumber\\
&& + \| (I - Q_n) \mathcal {K}_m'  (\varphi_m) (Q_n \varphi_m - \varphi_m )\|_\infty. 
\end{eqnarray}
Recall from (\ref{eq:2.7}) that
\begin{eqnarray}\nonumber
y_n = \mathcal{K}_m (Q_n\varphi_m) - \mathcal{K}_m (\varphi_m) - \mathcal {K}_m'  (\varphi_m) (Q_n \varphi_m - \varphi_m )  = 
R  ( Q_n \varphi_m - \varphi_m ),
\end{eqnarray}
Now proceeding as in the proof of Theorem 4.2,  we obtain
\begin{align*}
\|y_n \|_\infty = \|R  ( Q_n \varphi_m - \varphi_m )\|_\infty = O\left( \max\left\{ \tilde{h}^2 , h^3 \right\} \right).
\end{align*}
Note that
\begin{align}\label{eq:4.12}
\|\mathcal{K}_m'(\varphi_m) (I - Q_n )y_n\|_\infty \leq C_7 \|y_n\|_\infty h = O\left(h \max\left\{ \tilde{h}^2, h^3 \right\} \right).
\end{align}
Using (\ref{eq:4.4}) it can be seen that
\begin{eqnarray*}
\left \|\mathcal {K}_m'  (\varphi_m)   (I - Q_n) \mathcal {K}_m'  (\varphi_m)   (I - Q_n) \varphi_m \right \|_{\infty} = O \left (h^4 \right ).
\end{eqnarray*}
Thus, from  \eqref{eq:4.11}, \eqref{eq:4.12} and the above estimate, we obtain 
\begin{eqnarray}\label{eq:4.18}
 \|\mathcal{K}_m'   (\varphi_m)  (  \mathcal {K}_m (\varphi_m) -  \tilde{\mathcal{K}}_n^M (\varphi_m)  )\|_\infty = O\left(h \max\left\{ \tilde{h}^2, h^3 \right\} \right).
\end{eqnarray}
We recall the following result from Proposition 3.5:
\begin{eqnarray}\nonumber
\left \|\tilde{\mathcal{K}}_n^M (z_n^M) - \tilde{\mathcal{K}}_n^M (\varphi_m)  - \left (  \tilde{\mathcal{K}}_n^M \right )' (\varphi_m) (z_n^M - \varphi_m) \right \|_\infty &\leq & C_8 \left \|z_n^M - \varphi_m \right \|_\infty^2 .\\\label{eq:4.19}
& = & O (\max \{\tilde{h}^{2}, h^{ 3} \}^2).
\end{eqnarray}
Note that
\begin{eqnarray}\nonumber
\left \| \mathcal{K}_m'   (\varphi_m) 
\left (\left (  \tilde{\mathcal{K}}_n^M \right )' (\varphi_m)  -  \mathcal{K}_m'   (\varphi_m) \right  ) \right \| = O (h). 
\end{eqnarray}
Hence
\begin{equation}\label{eq:4.20}
\left \|\mathcal{K}_m'   (\varphi_m) \left (\left (  \tilde{\mathcal{K}}_n^M \right )' (\varphi_m)  -  \mathcal{K}_m'   (\varphi_m) \right  ) ( z_n^M - \varphi_m ) \right \|_\infty 
 = O (h \max \{\tilde{h}^2, h^3 \}),
\end{equation}
We thus obtain the following estimate using (\ref{eq:3.14}), (\ref{eq:4.18}), (\ref{eq:4.19}) and (\ref{eq:4.20}):
\begin{eqnarray*}
\|\mathcal {K}_m' (\varphi_m) (z_n^M - \varphi_m) \|_\infty = O (h \max \{\tilde{h}^2, h^3 \}).
\end{eqnarray*}
From (\ref{eq:4.10}) it follows that
\begin{eqnarray*}
\|\tilde{z}_n^M - \varphi_m \|_\infty = O (h \max \{\tilde{h}^2, h^3 \}).
\end{eqnarray*}
Since
$ \|\varphi - \varphi_m \|_\infty = O (\tilde{h}^2),$
the required result  follows. 
\end{proof}
\begin{remark}
It can be  shown that
\begin{equation}\label{**}
 \|z_n^G - \varphi \|_\infty = O \left (h \right ), 
 \;\;\; \|z_n^S - \varphi \|_\infty = O \left ({h}^{2} \right ).
\end{equation}
\end{remark}

\setcounter{equation}{0}
\section{Numerical Results}

For the sake of illustration, we quote some  numerical results from Grammont et al \cite{Gram3}  for the following example considered
 in Atkinson-Potra \cite{AtkP1}.

Consider
\begin{equation}\label{eq:5.1}
 x  (s) - \int_0^1 \kappa (s, t) \left [ f (t, x (t) \right ] d t  = \int_0^1 \kappa (s, t) z (t) d t, \;\;\; 0 \leq s \leq 1,
 \end{equation}
where 
$$
\kappa (s,t) = \left\{ {\begin{array}{ll}
(1 - s) t, &  0 \leq t \leq s \leq 1, \\
s ( 1 - t), &  0 \leq s \leq t \leq 1,
\end{array}}\right. \;\;\; \mbox{and} \;\;\; f (t, u) = \frac {1} {1 + t + u}
$$
with $z (t)$ so chosen that
$$ \varphi (t) = \frac { t (1 - t)} { t + 1}$$
is the solution of (\ref{eq:5.1}). 
In this example,
 $r$ can be chosen as large as we want. 
 
\subsection{ Piecewise Constant functions ($ r = 0$)}

Let $\mathcal{X}_n$ be the space of piecewise constant functions
with respect to the partition (\ref{eq:2.12}) and 
$ Q_n: L^\infty [0, 1] \rightarrow \mathcal{X}_n$ be the
discrete orthogonal projection defined by (\ref{eq:4.4})-(\ref{eq:4.6}).
The numerical quadrature is chosen to be the  composite Gauss $2$ rule with respect to partition (\ref{eq:2.1}) with 
$m = n^2$ subintervals. Then $\tilde{h} = h^2.$

In the following table, $\delta_G, \; \delta_S, \; \delta_M $ and $\delta_{IM}$ denote the computed orders of convergence in the
discrete Galerkin, discrete iterated Galerkin, discrete Modified Projection and the discrete iterated Modified Projection methods, respectively.  It can be seen that the computed values of order of convergence 
match well with the theoretically predicted values in (\ref{eq:4.6}), (\ref{eq:4.9}) and (\ref{**}).

\begin{center}
Table 5.1

\bigskip

\begin{tabular} {|c|cc|cc|cc|cc|}\hline
$n$ & $\| \varphi  - z_n^G \|_\infty$ & $\delta_G$ & $\| \varphi - z_n^S \|$ & $\delta_S$ 
& $\| \varphi - z_n^M \|_\infty$ & $\delta_M$  & $\| \varphi - \tilde{z}_n^M \|_\infty$ & $\delta_{IM}$\\
\hline
2&  $  1.22 \times 10^{-1} $ &              &  $ 8.40 \times 10^{-3} $ &                &   $  4.34 \times 10^{-3} $  &            &    $  5.23 \times 10^{-3} $ & \\
4&  $  8.65 \times 10^{-2} $ & $ 0.49 $ & $  2.35 \times 10^{-3} $ & $ 1.84 $  &   $ 4.31 \times 10^{-4}   $ & $ 3.33 $  &    $  3.14 \times 10^{-4} $ &  $ 4.06$\\
8&  $  5.09 \times 10^{-2} $ & $ 0.77 $ & $ 6.22 \times 10^{- 4}  $ &  $ 1.92 $ &   $ 5.28 \times 10^{- 5 }  $ &   $ 3.03 $  &    $  1.89 \times 10^{-5} $ & $4.05$\\
16& $ 2.70 \times 10^{-2} $ & $ 0.91 $ & $ 1.59 \times 10^{-4}  $ & $ 1.96 $  &   $ 6.92 \times 10^{- 6}  $ & $ 2.93 $  &    $  1.36 \times 10^{-6} $ & $3.80$\\
32& $ 1.33 \times 10^{-2}  $ & $ 1.02 $ & $ 4.02 \times 10^{-5}  $ &  $ 1.98 $ &   $ 8.38 \times 10^{-  7}   $ & $ 3.05 $   &    $  4.55 \times 10^{-8} $ &  $4.90$\\\hline
\end{tabular}
\end{center}
 \subsection{Piecewise Linear Functions ($ r = 1$)}

Let $\mathcal{X}_n$ be the space of piecewise linear polynomials 
with respect to the partition (\ref{eq:2.12}) and 
$ Q_n: L^\infty [0, 1] \rightarrow \mathcal{X}_n$ be the
discrete orthogonal projection defined by (\ref{eq:2.15}).
The numerical quadrature is chosen to be the composite Gauss 2 point rule with $ n^2$ intervals for the Galerkin and the iterated Galerkin 
method and the composite Gauss 2 point rule with $ n^3$ intervals for the modified projection and the iterated modified projection methods. In the latter case $\tilde{h}^2 = h^6.$ As a consequence,
it follows from   (\ref{eq:3.11}),  (\ref{*})  and (\ref{eq:3.19}) that the expected orders of convergence
in the discrete Galerkin, the discrete iterated Galerkin, the discrete modified projection and the discrete iterated modified projection methods are 
 $2, 4, 4$ and $6,$ respectively. The computational results given below match well with these orders.

\begin{center}
Table 5.2

\bigskip

\begin{tabular} {|c|cc|cc|cc|cc|}\hline
$n$ & $\| \varphi  - \varphi_n^G \|_\infty$ & $\delta_G$ & $\| \varphi - \varphi_n^S \|$ & $\delta_S$ 
& $\| \varphi - \varphi_n^M \|_\infty$ & $\delta_M$  & $\| \varphi - \tilde{\varphi}_n^M \|_\infty$ & $\delta_{IM}$\\
\hline
2&  $  1.32 \times 10^{-1} $ &              &  $ 4.97 \times 10^{-3} $ &                &   $  1.54 \times 10^{-3} $  &            &    $  1.34 \times 10^{-3} $ & \\
4&  $  4.98 \times 10^{-2} $ & $ 1.41 $ & $  4.46 \times 10^{-4} $ & $ 3.48 $  &   $ 1.12 \times 10^{-4}   $ & $ 3.78 $  &    $  1.89 \times 10^{-5} $ &  $ 6.15 $\\
8&  $  1.58 \times 10^{-2} $ & $ 1.66 $ & $ 3.89 \times 10^{- 5}  $ &  $ 3.52 $ &   $ 1.06 \times 10^{- 5 }  $ &   $ 3.40 $  &    $  2.48\times 10^{-7} $ & $6.25$\\
16& $ 4.51 \times 10^{-3} $ & $ 1.81 $ & $ 3.15 \times 10^{- 6}  $ & $ 3.62 $  &   $ 9.10 \times 10^{- 7 } $ &  $ 3.54$&  $  2.92\times 10^{-9} $   &  $6.41$\\
\hline
\end{tabular}
\end{center}



\begin{thebibliography}{999}



\bibitem{Atk1}
K. E. ATKINSON, The numerical evaluation of fixed points for completely continuous operators, 
SIAM J. of Numerical Analysis, 10 (1973), 799 - 807.

\bibitem{Atk2}
K. E. ATKINSON, The Numerical Solutions of Integral Equations of the 
Second Kind, Cambridge University Press, Cambridge, U.K., 1997.


\bibitem{AtkG}
K. ATKINSON, I. GRAHAM and I. SLOAN, Piecewise continuous collocation
for Integral Equations, SIAM J. of Numerical Analysis, 20 (1983), pp. 172-186.


\bibitem{AtkP1}
 K. E. ATKINSON and F. A. POTRA, 
 Projection and iterated projection methods for nonlinear integral equations, 
SIAM J. Numer. Anal., 24 (1987), 1352 - 1373.

\bibitem{AtkP2}
 K. E. ATKINSON and F. A. POTRA, 
 The discrete Galerkin method for nonlinear integral equations, J. Integral Equations Appl.,
1 (1988), 17 - 54.

 


\bibitem{Gram1} L. GRAMMONT,
A Galerkin's perturbation type method to approximate a fixed point of a compact operator, 
 International Journal of Pure and Applied Mathematics, 69 (2011), 1-14.
 
 \bibitem{Gram2}
L. GRAMMONT and R. P. KULKARNI,
 A superconvergent projection method for nonlinear compact operator equations,
C. R. Acad. Sci. Paris,   342 (2006),  215-218.


\bibitem{Gram4}
L. GRAMMONT,  R. P. KULKARNI AND P.  B. VASCONCELOS,
Modified projection and the iterated modified projection methods for  nonlinear integral 
equations,  J. Integral Eqns. Appl., 25 (2013), 481-516.



\bibitem{Gram3} L. GRAMMONT, R. P. KULKARNI and T. J. NIDHIN,
Modified projection method for  Urysohn integral 
equations with non-smooth kernels, Journal of Computational and Applied Mathematics,  294 (2016),
309-322.

\bibitem{Kra}
M. A. KRASNOSELSKII, 
 Topological Methods in the Theory of Nonlinear Integral Equations,
Pergamon Press, London,  1964.


\bibitem{KraV}
M. A. KRASNOSELSKII, G. M. VAINIKKO, P. P. ZABREIKO, Ya. B. RUTITSKII and V. Ya.STETSENKO,
 Approximate Solution of   Operator Equations,
P. Noordhoff, Groningen, 1972.

\bibitem{KraZ} M. A. KRASNOSELSKII AND P. P. ZABREIKO,
Geometrical Methods of
Nonlinear Analysis,  Springer-Verlag, Berlin, 1984.


\bibitem{Kul1} R. P. KULKARNI and G. RAKSHIT,
Discrete modified projection method for Urysohn integral equations with smooth kernels,
Applied Numerical Mathematics,  126 (2018), 180-198.



\end{thebibliography}
\end{document}